\documentclass{article}
\usepackage{amsmath, amsthm, amssymb}
\usepackage{color}
\theoremstyle{plain}
\newtheorem{theorem}{Theorem}[section]
\newtheorem{lemma}[theorem]{Lemma}
\newtheorem{proposition}[theorem]{Proposition}
\newtheorem{corollary}[theorem]{Corollary}
\newtheorem{example}[theorem]{Example} 
\theoremstyle{definition}
\newtheorem{remark}[theorem]{Remark}

\newenvironment{proofof}[1]{\noindent{\it Proof of
#1.}}{\hfill$\square$\\\mbox{}}

\DeclareMathOperator{\Ind}{Ind}

\usepackage{authblk}

\usepackage{url}

\DeclareMathOperator{\F}{\mathbb{F}}

\begin{document}
\date{}
\title{Lower bounds on the Noether number}
\author{K. Cziszter
\thanks{Email: \texttt{cziszter.kalman@gmail.com}\\
Partially supported by  National Research, Development and Innovation Office, NKFIH   grants PD113138, ERC~HU~15 118286 and K115799.}
}
\author{M. Domokos 
\thanks{Email:  \texttt{domokos.matyas@renyi.mta.hu}\\
Supported by National Research, Development and Innovation Office,  NKFIH K 119934.}}
\affil{MTA Alfr\'ed R\'enyi Institute of Mathematics, Re\'altanoda utca 13-15, 1053 Budapest,  Hungary}
\maketitle
\begin{abstract} 
The best known method to give a lower bound for the Noether number of a given finite group is to use the fact that it is greater than or equal to the Noether number of any of the subgroups or factor groups. 
The results of the present paper show in particular that these inequalities are strict for proper subgroups or factor groups. This is established by studying the algebra of coinvariants 
of a representation induced from a representation of a subgroup.  
\end{abstract}




\def\K{\mathbb{F}}
\def\td{\mathrm{topdeg}}


\section{Introduction} \label{sec:intro} 

Throughout this paper $G$ is a finite group, $\F$ is a field whose characteristic does not divide the order of $G$. Given a finite dimensional $\F G$-module $W$ we write $S(W)$ for the symmetric tensor algebra of $W$. The linear action of $G$ on $W$ extends to an action via $\F$-algebra automorphisms of $S(W)$. 
We are interested in the subalgebra 
\[S(W)^G=\{f\in S(W)\mid g\cdot f=f\quad \forall g\in G\}\] 
of {\it $G$-invariants}.  The {\it Noether number} $\beta(G,W)$ which is the smallest number $d$ such that $S(W)^G$ is generated as an algebra by  its elements of degree at most $d$.
 A fundamental result in the invariant theory of finite groups is that for $\beta(G): =\sup\{\beta(G,W) \mid W\text{ is an }\F G\text{-module}\}$ we have the inequality 
(known as the {\it Noether bound}) 
\[\beta(G)\le |G|  \quad \text{(see \cite{noether}, \cite{fleischmann}, \cite{fogarty})}.\]  
Improvements of the Noether bound or exact values of the Noether number can be found in \cite{schmid}, \cite{domokos-hegedus}, 
\cite{sezer}, \cite{CzD:1}, \cite{CzD:2}, \cite{Cz1}, \cite{hegedus-pyber}, \cite{CzDG}, \cite{CzDSz}. 
Our starting point is the following observation of  B. Schmid \cite{schmid}:

\begin{lemma}\label{lemma:schmid} {\rm [Schmid]} Let $H$ be a subgroup of $G$ and let $V$  be an $\F H$-module. We have the inequality 
\begin{align}\label{eq:schmid1} 
\beta(G, \Ind_H^G V ) \ge \beta(H,V)
\end{align}
and consequently 
\begin{align}\label{eq:schmid2} 
\beta(G)\ge \beta(H). 
\end{align}
\end{lemma}
This lower bound also has an obvious counterpart for  homomorphic images. 
Indeed, as any $\F(G/N)$-module $W$ can be interpreted as an $\F G$-module $W$ on which $N$ acts trivially, we have $\beta(G,W) = \beta(G/N, W)$ and consequently
\begin{align}\label{factor_triv}
\beta(G) \ge \beta(G/N).
\end{align}
The inequalities \eqref{eq:schmid2}, \eqref{factor_triv} and their improvements  are the main tools to produce lower bounds for 
$\beta(G)$.  In particular, $\beta(G)$ is not smaller than the maximal order of an element of $G$. 
The inequality \eqref{eq:schmid1} is sharp in the sense that it may happen  for 
some group $G$ and a proper subgroup $H\neq G$ that we have $\beta(G, \Ind_H^G V ) = \beta(H,V)$,  see Example~\ref{example:dihedral}.  
It was observed, however, in  \cite{CzDSz} that for the groups $G$ of order less than $32$ 
and for some other infinite classes of groups neither \eqref{eq:schmid2} nor \eqref{factor_triv} are sharp.  As a result of our inquiry we can now prove that this is a general phenomenon: 

\begin{theorem}\label{thm:SF}
For any proper subgroup $H \subsetneq G$  we have
\begin{align}\label{eq:3thm}
\beta(G)& >\beta(H)
\end{align} 
and
for any normal subgroup $N \triangleleft G$ we have
\begin{align}\label{eq:beta+beta} \beta(G) \ge \beta(N) + \beta(G/N) -1. 
\end{align}
\end{theorem}

\begin{example} {\rm 
Inequality \eqref{eq:beta+beta} is sharp  as it is shown by the following examples where  \eqref{eq:beta+beta} holds with equality: 
\begin{enumerate} 
\item For the non-abelian semidirect product $G=C_5\rtimes C_4$ we have $\beta(G)=8=5+4-1=\beta(C_5)+\beta(G/C_5)-1$ by \cite[Proposition 3.2]{CzD:1}. 
\item   For the non-abelian semidirect product $G=C_p\rtimes C_3$ 
(where $p$ is a prime  congruent to $1$ modulo $3$) 
we have $\beta(G)=p+3-1=\beta(C_p)+\beta(G/C_p)-1$ 
by \cite{Cz1}.  
\item For a divisor $m$ of $n$ we have $\beta(C_n\oplus C_m)=
n+m-1=\beta(C_n)+\beta(C_m)-1$ by classical results on the Davenport constant, see for example \cite{CzDG} for  a survey on connections 
between the Noether number and the Davenport constant (studied extensively in  arithmetic  combinatorics). 
\end{enumerate} 
Inequality  \eqref{eq:3thm} is  also sharp, as  by \cite{CzD:2} we know that  if $H$ is a cyclic subgroup of  index  $2$ in $G$, and $G$ is not cyclic or dicyclic, then 
$\beta(G)=\beta(H)+1$. }
\end{example}

Theorem~\ref{thm:SF} is obtained by studying the top degree of the coinvariant algebra $S(W)_G$, so let us recall the relevant definitions first.  
Note that $S(W)=\bigoplus_{d=0}^\infty S(W)_d$ is graded, $S(W)_0=\F\subset S(W)$, and the degree $1$ homogeneous component is 
$S(W)_1=W\subset S(W)$. The $G$-action preserves the grading. We shall deal with commutative graded $\F$-algebras $R=\bigoplus_{d=0}^\infty R_d$ such that $R_0=\F$, 
and  we shall denote by $R_+=\bigoplus_{d>0}R_d$ the ideal spanned by the homogeneous components of positive degree. For a graded vector space $X=\bigoplus_{d=0}^\infty X_d$ we set  
\[\td(X)=\sup\{d\mid X_d\neq 0\}.\] 

The {\it Hilbert ideal} in $S(W)$ is the ideal $S(W)_+^GS(W)$ generated by the homogeneous invariants of positive degree, and the corresponding factor algebra 
\[S(W)_G=S(W)/S(W)_+^GS(W)\] 
is called the {\it algebra of coinvariants}. 
Our results will concern the following quantity associated with the $\F G$-module $W$: 
\begin{align*}   
b(G,W)=\td(S(W)_G).
\end{align*} 
Note that $b(G,W)$ is the minimal non-negative integer $d$ such that 
the $S(W)^G$-module $S(W)$ is generated by homogeneous elements of degree at most $d$. 
Following \cite{kohls-sezer} and \cite{CzDG} we introduce also 
\[b(G)=\sup\{b(G,W) \mid W\text{ is an }\F G\text{-module}\}.\]
Remark that by  the graded Nakayama lemma  $\beta(G,W)$ can also be recovered as the top degree of a certain finite dimensional algebra, namely 
\begin{align*}\beta(G,W)=\td(S(W)_+^G/(S(W)_+^G)^2). 
\end{align*}

Our first main result shows that the Noether number is always strictly monotone on subgroups: 
 
\begin{theorem} \label{thm:betab}
Let $H\subsetneq G$ be a proper subgroup of $G$ and let $V$ be an $\F H$-module. 
Then the inequality
\begin{equation}\label{eq:1thm}
b(G, \Ind_H^G V ) \ge \beta(H,V)
\end{equation}
holds. 
In particular, we have the  inequality 
\begin{equation}\label{eq:2thm} b(G)\ge \beta(H). 
\end{equation} 
 \end{theorem}

Our second main result is the following finer statement for the case of a normal subgroup: 

\begin{theorem}\label{thm:G/N} Let $N$ be a normal subgroup of $G$, 
$U$ an $\F (G/N)$-module and $V$ an $\F N$-module.  Then we have the inequality 
\begin{align}\label{eq:b+b} b(G, U\oplus \Ind_N^G V) \ge b(G/N,U) + b(N,V).
\end{align} 
\end{theorem}

To see how  these two theorems imply Theorem~\ref{thm:SF} the key step is the following result from \cite{CzD:3}:

\begin{lemma}\label{lemma:beta=b+1} 
We have the equality 
\[\beta(G)=b(G)+1.\] 
\end{lemma} 
\begin{proof} 
The inequality $\beta(G,W)\le b(G,W)+1$ for any $W$ is a consequence of the existence of the Reynolds operator $\tau^G:S(W)\to S(W)^G$ given for a linear action of a finite group $G$ on an $\F$-vector space $X$ by the formula 
\[\tau^G(x)=\frac{1}{|G|}\sum_{g\in G} g\cdot x \quad (x\in X)\] 
(see for example the proof of Corollary 3.2 in \cite{CzD:3} for the details). Hence we have the inequality $\beta(G)\le b(G)+1$. 
On the other hand, Lemma 3.1 in \cite{CzD:3} asserts in particular that for any $\F G$-module $W$ there exists an $\F G$-module $Z$ such that $\beta(G,Z)\ge b(G,W)+1$. 
This clearly implies the reverse inequality $\beta(G)\ge b(G)+1$. 
\end{proof}

The paper is organized as follows.
Theorem~\ref{thm:betab} is proved in Section~\ref{sec:subgroup}, and Theorem~\ref{thm:G/N} is proved in Section~\ref{sec:lower}. 
For an arbitrary positive integer $k$ the $k$th Noether number 
$\beta_k(G)$ was introduced in \cite{CzD:1} where it was shown that 
$\beta(G)\le \beta_{|G:H|}(H)$ for any subgroup $H$ of $G$ and 
$\beta(G)\le \beta_{\beta(G/N)}(N)$ for any normal subgroup $N$ of $G$. 
These results can be very efficiently applied to obtain good bounds for the Noether number of $G$ from the $k$th Noether numbers of its subquotients, see for example \cite{CzDSz}. 
It seems worthwhile therefore to extend Theorem~\ref{thm:betab} and 
Theorem~\ref{thm:G/N} for the $k$th Noether number. This is done in Section~\ref{sec:k-th_noether}.


\section{Lower bound in terms of subgroups}\label{sec:subgroup}

Take a proper subgroup $H$ of $G$. Choose a system 
$\mathcal C$ of left $H$-coset representatives in $G$. We shall assume that $1\in\mathcal C$. Let $W$ be an $\F G$-module containing an $\F H$-submodule $V$ such that 
$ W = \bigoplus_{g \in \mathcal C} g\cdot V$. That is, $W\cong \mathrm{Ind}^G_H(V)$, the 
$\F G$-module induced from the $\K H$-module $V$. 
The projection $\pi: W \to V$ with kernel $\bigoplus_{g\in\mathcal{C}\setminus \{1\}}g\cdot V$ extends to an  $\F$-algebra surjection $\pi: S(W)\to S(V)$ from the symmetric tensor algebra $S(W)$ onto its subalgebra $S(V)$. Clearly $\pi$ is $H$-equivariant and is degree preserving. Equality \eqref{eq:schmid1} in Lemma~\ref{lemma:schmid} is a consequence of the following: 
\begin{align} \label{eq:schmid0} 
\pi(S(W)^G) = S(V)^H. 
\end{align}

\begin{example}\label{example:dihedral} 
{\it Equality may hold in \eqref{eq:schmid1} even if $H\neq G$:} 
{\rm Let $G$ be the dihedral group of order $2n$, and let $H$ be its cyclic index two subgroup consisting of the rotations. 
Let $W$ be any irreducible $2$-dimensional $\F G$-module. Then $W=\Ind_H^G V$, where $V$ is the $1$-dimensional $\F H$-module on which the generators of $H$ acts via multiplication by a primitive $n$th root of unity. It is well known that $S(V)^H$ is generated by a single invariant of degree $n$, whereas $S(W)^G$ is generated by  homogeneous invariants of degree $2$ and $n$. Therefore $\beta(W,G)=n=\beta(V,H)$ in this case.  }
\end{example} 

Equality \eqref{eq:schmid0} implies that 
the Hilbert ideal $S(W)^G_+S(W)$ in $S(W)$ is mapped by $\pi$ into the Hilbert ideal $S(V)^H_+S(V)$ in $S(V)$, whence  we have an induced graded $\F$-algebra epimorphism $S(W)_G \to S(V)_H$ between the corresponding algebras of coinvariants.
This shows that
\begin{align}\label{eq:b}
b(G,\Ind_H^G V ) \ge b(H,V).
\end{align}
The main result of this section is the strengthening \eqref{eq:1thm} in Theorem~\ref{thm:betab} of \eqref{eq:b}.  
In order to prove it we shall consider the factor algebra 
$R=S(W)/J$  where $J$ is the ideal of $S(W)$ generated by the set of quadratic elements 
\[\{(g\cdot v)(g'\cdot v')\mid v,v'\in V,\ g,g'\in\mathcal C,\  g\neq g'\}.\] 
Denote by $\eta:S(W)\to R$ the natural surjection. 
Since $J$ is a $G$-stable homogeneous ideal, the algebra $R$ inherits from $S(W)$ a grading and a $G$-action via degree preserving $\F$-algebra automorphisms, so that 
$\eta$ is $G$-equivariant and preserves the degree. For each $g\in\mathcal C$ the subspace $S(g\cdot V)_+ \subset S(W)$ is mapped by $\eta$  isomorphically to an ideal  $I^{(g)}$ of $R$.  Obviously $R_+=\bigoplus_{g\in \mathcal C}I^{(g)}$ is the ring theoretic direct sum of these ideals, and  
\[R=\F \oplus \bigoplus_{g\in \mathcal C}I^{(g)},\] 
where the ideals $I^{(g)}$ ($g\in \mathcal{C}$) annihilate each other, the direct summand $\F$ is a subring of $R$ containing the identity element of $R$, and $\F=R_0$ is the degree zero homogenous component of the graded $\F$-algebra $R$.  
Moreover, for each $g\in \mathcal C$ the restriction of $\eta$ to the subalgebra 
$S(g\cdot V)\subset S(W)$ is an isomorphism 
\[\eta\vert_{S(g\cdot V)}:S(g\cdot v)\stackrel{\cong}\longrightarrow \F\oplus I^{(g)}\]  
of graded algebras.  
For ease of notation write $T$ for the subalgebra 
\[T=\F\oplus I^{(1_G)}\subset R.\] 
Then  
\[\eta\vert_{S(V)}:S(V)\stackrel{\cong}\longrightarrow T\] 
is an $H$-equivariant isomorphism of graded $\F$-algebras.

\begin{proposition} \label{prop:complement}
We have $R_+^G R \cap T \subseteq T_+^H T_+$.
\end{proposition}

\begin{proof} 
Consider an arbitrary element $r \in R_+^G R$. Then  $r =\sum_i x_i \tau^G(y_i) + \sum_j \tau^G (z_j)$ for some $x_i,y_i,z_j \in R_+$, since $\tau^G:R_+\to R_+^G$ is surjective.
Observe now that any $x \in R_+$ can be expressed in the form $x = \sum_{g \in \mathcal C} g \cdot t_g$ where $t_g \in T_+$ for all $g \in \mathcal C$.
After expanding each $x_i,y_i,z_j$ in this form and then using the linearity of $\tau^G$ and the fact that $\tau^G(g \cdot t) = \tau^G(t)$ for any $g \in G$
we get an expression
\begin{align} \label{eloall}
r=\sum_{i\in\Lambda} (g_i \cdot u_i)\tau^G(v_i) + \sum_{j\in\Gamma}
\tau^G(w_j)
\quad \text{ where } u_i,v_i, w_j \in T_+, g_i \in \mathcal C.
\end{align} 
Note that here 
\begin{align} \label{eq:16}
(g_i \cdot u_i)\tau^G(v_i) = \frac{1}{|G:H|}(g_i \cdot u_i)(g_i \cdot \tau^H(v_i)) 
= \frac{1}{|G:H|}g_i \cdot (u_i \tau^H(v_i)).
\end{align}

Now assume in addition that $r \in T$. This means that for any $g \in \mathcal C \setminus H$ the terms in the sum \eqref{eloall} belonging to $ g\cdot T$ cancel each other. 
By gathering together all these terms we get  for each  $g  \in \mathcal C \setminus H$ the equation  
\[
0 =\sum_{i\in\Lambda: g_i =g} g \cdot (u_i\tau^H(v_i))  + g \cdot \sum_{j\in\Gamma} \tau^H(w_j) .
\]
After multiplying this equality from the left by $g^{-1}$ we conclude that in fact 
\[\sum_{j\in\Gamma}\tau^H(w_j) 
=\sum_{i\in\Lambda: g_i =g} u_i\tau^H(v_i)
\in T_+^HT_+\] 
(in this step we use that $H$ is a proper subgroup of $G$, so there exists an element $g\in \mathcal C \setminus H$).
Finally, after gathering together all terms in \eqref{eloall} belonging to $T$ we get
\begin{align*}
r =\frac{1}{|G:H|}\left( \sum_{i\in\Lambda: g_i \in H}  u_i\tau^H(v_i) + \sum_{j\in\Gamma} \tau^H(w_j)
\right) 
\in T_+^H T_+ & \qedhere. 
\end{align*}
\end{proof}

\begin{corollary}\label{cor:felemelt} 
We have $(S(W)_+^GS(W))\cap S(V)\subseteq S(V)_+^HS(V)_+$. 
\end{corollary}
\begin{proof} To simplify notation set $M=S(W)$ and $N=S(V)$. 
We get from Proposition~\ref{prop:complement} that
\begin{align*}
M_+^G M \cap N &\subseteq 
\eta^{-1}(\eta(M_+^G M \cap N)) \subseteq \eta^{-1}(\eta(M_+^G M) \cap \eta (N))) \\
& = \eta^{-1}(R_+^GR \cap T) \subseteq \eta^{-1}(T_+^HT_+) 
= N_+^H N_++\ker(\eta). 
\end{align*}
Since $N\cap \ker(\eta)=(0)$, we conclude that 
\[M_+^GM\cap N\subseteq N_+^HN_+. \qedhere \]
\end{proof}

\bigskip
\begin{proofof}{Theorem~\ref{thm:betab}}
We have $M_+^GM\cap N^H\subseteq N_+^HN_+$ as an immediate consequence of Corollary~\ref{cor:felemelt}, whence applying the $N^H$-module homomorphism $\tau^H$ we conclude $M_+^GM\cap N^H\subseteq (N_+^H)^2$. 
Denote by $\kappa$ the canonical surjection $\kappa: S(W)\to S(W)_G=M/M_+^GM$. 
The kernel of the restriction of $\kappa$ to $N^H$ is 
\[\ker(\kappa\vert_{N^H})=N^H\cap M_+^GM\subseteq (N_+^H)^2.\]  
It follows that the natural surjection $\nu:N^H\to N^H/(N_+^H)^2$ factors through 
$\kappa\vert_{N^H}$; that is, there exists a graded $\F$-algebra homomorphism 
$\gamma:\kappa(N^H)\to  N^H/(N_+^H)^2$ such that 
$\nu=\gamma\circ \kappa\vert_{N^H}$. 
In particular, the algebra $N^H/(N_+^H)^2$ is a homomorphic image of the subalgebra 
$\kappa(N^H)$ of the coinvariant algebra $S(W)_G$. 
Consequently, we have the inequalities 
\[b(G,W)=\td(S(W)_G)\ge \td(\kappa(N^H))\ge \td(N^H/(N_+^H)^2)=\beta(H,V)\]
that show \eqref{eq:1thm}. 
Applying \eqref{eq:1thm} to an $\F H$-module $V$ for which $\beta(H)=\beta(H,V)$ we obtain \eqref{eq:2thm}, which together with Lemma~\ref{lemma:beta=b+1} in turn imply 
\eqref{eq:3thm}. 
\end{proofof}

\begin{remark} Combining \eqref{eq:1thm} with \eqref{eq:b} we  have in fact the inequality 
\begin{align}b(G, \Ind_H^G V ) \ge \max \{ \beta(H,V), b(H,V) \}.  
\end{align}
\end{remark}


\section{Normal subgroups} \label{sec:lower}
Let $N$ be a normal subgroup of $G$. 
Given  an $\F (G/N)$-module $U$ and an $\F N$-module $V$ let us consider the $\F G$-module 
\begin{align}\label{eq:rep}W := U \oplus \Ind_N^G V\end{align}
where we view $U$ as an $\F G$-module on which $N$ acts trivially. 
The relative Reynolds operator is 
defined as 
\[\tau_N^G: S(W)^N \to S(W)^G, \quad \tau_N^G(m) = \frac{1}{|G:N|}\sum_{g \in \mathcal{C}}  m^g\] 
where $\mathcal{C}$ is a system of $N$-coset 
representatives in $G$,  
$m \in S(W)^N$, and we write $m^g$ for $g^{-1}\cdot m$. The map $\tau_N^G$ is an $S(W)^G$-module homomorphism and is 
surjective onto $S(W)^G$. Moreover, we have $\tau_N^G\circ\tau^N=\tau^G$.  
Note that the direct sum decomposition $W=U\oplus\bigoplus_{g\in\mathcal{C}}g\cdot V$ induces an identification 
\[S(W)=S(U)\otimes \bigotimes_{g\in\mathcal{C}}g\cdot V,\] 
and $S(U)$, $S(V)$, $S(g\cdot V)$, will be considered as subalgebras of $S(W)$ in the obvious way.  
Let 
\[\pi : S(W) \to S(U)\otimes S(V)=S(U \oplus V)\] 
be the $N$-equivariant $\F$-algebra epimorphism  of graded algebras whose kernel is the ideal generated by $\sum_{g\in \mathcal{C}\setminus N}g\cdot V$, and $\pi$ is the identity map on the subalgebra $S(U\oplus V)$ of $S(W)$.

\begin{lemma}\label{lemma:image}
The image of the Hilbert ideal $S(W)_+^GS(W)$ under $\pi$ is generated as an ideal in 
$S(U\oplus V)$ by $S(U)_+^{G/N}$ and $S(V)_+^N$. 
\end{lemma}

\begin{proof}
As an $\F$-vector space $S(W)^G_+$ is spanned by elements of the form 
$\tau^G(w)$ where $w$ ranges over any $\F$-vector space basis of  $S(W)_+$.  
Now $S(W)_+$ is spanned by elements of the form  
$w=uv$ where $u \in S(U)$ is homogeneous and 
$v=v_1^{g_1}\cdots v_r^{g_r}$, where $r\in \mathbb{N}_0$, $g_i\in G$, $v_i\in V\subset S(W)_1$, and $\deg(u)>0$ or $r=\deg(v)>0$.  
Assume that $w$ is of this form.  
Then 
\[\tau^G(w) = \tau_N^G(\tau^N (w))= \tau_N^G(u \tau^N(v)),\] 
and hence we have 
\[ \pi(\tau^G(w)) = \frac{1}{|G:N|}\sum_{g \in \mathcal{C}} \pi( u^g \tau^N(v) ^g) = 
\frac{1}{|G:N|}\sum_{g \in \mathcal{C}} u^g \pi(\tau^N(v) ^g). \]
Now $\pi(\tau^N(v) ^g)= \pi(\tau^N(v^g))\neq 0$ if and only if $v^g \in S(V)\subset S(W)$. 
As a result we get
\begin{align}\label{eq:lenyeg}
\pi(\tau^G(w)) =
\begin{cases}0 & \text{ if }v^g\notin S(V)\text{ for all }g\in\mathcal{C} \\
\frac{1}{|G:N|} u^g \tau^N(v^g) & \text{ if } v^g\in S(V)_+ \text{ for some }g\in\mathcal{C}\\
\tau^{G/N}(u) & \text{ if } v = 1
\end{cases}.
\end{align}
The elements on the right hand side of \eqref{eq:lenyeg} all belong to the ideal 
$I$ generated by $S(U)_+^{G/N}$ and $S(V)_+^N$, implying  that $\pi(S(W)_+^G)S(W)\subseteq I$. For the reverse inclusion note that $\pi(S(U)_+^{G/N})=S(U)_+^{G/N}\subset S(W)_+^G$, and for any $v\in S(V)_+^N$ we have $v=\pi(\sum_{g\in \mathcal{C}}v^g)\in 
\pi(S(W)_+^G)$. 
\end{proof}

\bigskip 
\begin{proofof}{Theorem~\ref{thm:G/N}} 
Consider the natural surjection $\rho:S(U)\otimes S(V)\to S(U)_{G/N}\otimes S(V)_N$. 
The kernel of $\rho$ is the ideal generated by $S(U)_+^{G/N}$ and $S(V)_+^N$, whence by Lemma~\ref{lemma:image} we have $\ker(\rho)=\pi(S(W)_+^GS(W))$. 
It follows that the Hilbert ideal $S(W)_+^GS(W)$ is contained in $\ker(\rho\circ \pi)$, hence 
$\rho\circ\pi$ factors through the natural surjection $S(W)\to S(W)_G$. 
Consequently there exists a degree preserving $\F$-algebra surjection 
$S(W)_G\to S(U)_{G/N}\otimes S(V)_N$. This obviously implies that 
\begin{align*}
\td(S(W)_G)\ge \td(S(U)_{G/N}\otimes S(V)_N)
\\=\td(S(U)_{G/N})+\td(S(V)_N),\end{align*} 
which is the desired inequality \eqref{eq:b+b}. 

The inequality \eqref{eq:beta+beta} follows from \eqref{eq:b+b} by 
Lemma~\ref{lemma:beta=b+1}. 
\end{proofof}

\begin{remark} Theorem~\ref{thm:G/N} in the special case when $G/N$ is abelian 
was proved in \cite[Theorem 4.3]{CzD:2}, and in the special case when $G$ is a 
direct product $N\times N_1$ it was proved in \cite[Theorem 3.4]{CzD:3}. 
\end{remark}


\section{The $k$th Noether number}\label{sec:k-th_noether} 

Given a $\F G$-module $W$ and a positive integer $k$ we set 
\[\beta_k(G,W)=\td(S(W)^G/((S(W)_+^G)^{k+1})\] 
and call 
\[\beta_k(G)=\sup\{\beta_k(G,W)\mid W\text{ is an }\F G\text{-module}\}\] 
the {\it $k$th Noether number}. In the special case $k=1$ we recover the Noether number. 
The study of this quantity began in \cite{CzD:1}, see \cite{CzDG} for a survey. 
Moreover, set 
\[b_k(G,W)=\td(S(W)/(S(W)_+^G)^kS(W))\] 
and 
\[b_k(G)=\sup\{b_k(G,W) \mid W\text{ is an }\F G\text{-module}\}.\] 
Again in the special case $k=1$ we recover $b(G,W)$ and $b_k(G)$. 
It was shown in \cite{CzD:3} that 
\[\beta_k(G,W)\le b_k(G,W)+1\ \text{ and }\ \beta_k(G)=b_k(G)+1.\] 

\begin{theorem} \label{thm:betab_k}
Let $H\subsetneq G$ be a proper subgroup of $G$ and let $V$ be an $\F H$-module. 
Then the inequality
\begin{align*}
b_k(G, \Ind_H^G V ) \ge \beta_k(H,V)
\end{align*}
holds. 
In particular, we have the  inequality 
\begin{align*}
b_k(G)\ge \beta_k(H)
\end{align*} 
and the strict inequality 
\begin{align*}
\beta_k(G)>\beta_k(H).
\end{align*} 
 \end{theorem}
\begin{proof}
We use the notation of Section~\ref{sec:subgroup}. First we claim that $(R_+^G)^kR\cap T\subseteq (T_+^H)^kT_+$. Similarly to the proof of Proposition~\ref{prop:complement}, 
any $r\in (R_+^G)^kR$ can be written as 
\begin{align}\label{eq:sumi}
r=\sum_{i\in\Lambda}(g_i\cdot u_i)\tau^G(v_i^{(1)})\dots\tau^G(v_i^{(k)})+
\sum_{j\in\Gamma} \tau^G(w_j^{(1)})\dots\tau^G(w_j^{(k)})
\end{align}  
where $g_i\in\mathcal{C}$, $u_i,v_i^{(l)},w_j^{(l)}\in T_+$. 
Take an element  $g  \in \mathcal C \setminus H$. 
It follows from \eqref{eq:sumi} and \eqref{eq:16} that  if $r\in T$, then 
\[0 = g \cdot \sum_{i\in\Lambda: g_i =g}u_i\tau^H(v_i^{(1)})\dots\tau^H(v_i^{(k)})  + 
g \cdot \sum_{j\in\Gamma} \tau^H(w_j^{(1)})\dots \tau^H(w_j^{(k)}) , 
\] 
implying that 
\[\sum_{j\in\Gamma} \tau^H(w_j^{(1)})\dots \tau^H(w_j^{(k)})\in (T_+^H)^kT_+.\]  
Therefore if $r\in (R_+^G)^kR\cap T$ then 
\begin{align*}r=\frac{1}{|G:H|^k}\left(\sum_{i\in\Lambda: g_i =1_G}u_i\tau^H(v_i^{(1)})\dots\tau^H(v_i^{(k)}\right)  + 
 \sum_{j\in\Gamma} \tau^H(w_j^{(1)})\dots \tau^H(w_j^{(k)}))
 \\\in (T_+^H)^kT_+ . 
\end{align*} 
Similarly to Corollary~\ref{cor:felemelt} we conclude that 
$(M_+^G)^kM\cap N\subseteq (N_+^H)^kN_+$, which immediately implies 
(using the Reynolds operator $\tau^H$) that 
\begin{align}\label{eq:N+}
(M_+^G)^kM\cap N^H\subseteq (N_+^H)^{k+1}.
\end{align}  
Denote by $\kappa$ the natural surjection $\kappa:M\to M/(M_+^G)^kM$. The inclusion 
\eqref{eq:N+} implies that there exists a graded $\F$-algebra surjection from the subalgebra 
$\kappa(N^H)$ of $M/(M_+^G)^kM$ onto $N/(N_+^H)^{k+1}$. Thus we have 
\[\td(M/(M_+^G)^kM)\ge\td(\kappa(N^H))\ge \td(N/(N_+^H)^{k+1}),\] 
yielding the desired inequality 
$b_k(G,W)\ge\beta_k(H,V)$.  
\end{proof}

\begin{theorem}\label{thm:G/N_k} Let $N$ be a normal subgroup of $G$, 
$U$ an $\F (G/N)$-module and $V$ an $\F N$-module.  Then for any positive integers $r,s$  we have the inequality 
\begin{align*}
b_{r+s-1}(G, U\oplus \Ind_N^G V) \ge b_r(G/N,U) + b_s(N,V).
\end{align*} 
In particular, we have 
\begin{align*}
\beta_{r+s-1}(G) \ge \beta_r(G/N) + \beta_s(N) -1.
\end{align*}
\end{theorem}

\begin{proof}
Set $I=S(U)_+^{G/N}S(U)\triangleleft S(U)$, 
$J=S(V)_+^NS(V)\triangleleft S(V)$, and 
$K=S(W)_+^GS(W)\triangleleft S(W)$. 
With the notation of Section~\ref{sec:lower} we have that 
$\pi(K)=(I,J)\triangleleft S(U)\otimes S(V)=S(U\oplus V)$ 
by Lemma~\ref{lemma:image}. 
Hence denoting by 
\[\rho:S(U)\otimes S(V)\to S(U)/I^r\otimes S(V)/J^s\] 
the natural surjections, we have 
\[\pi(K^{r+s-1})=(I,J)^{r+s-1}\subseteq (I^r,J^s)=\ker\rho.\] 
It follows that there exists a degree preserving $\F$-algebra surjection 
\[S(W)/K^{r+s-1}\to S(U)/I^r\otimes S(V)/J^s,\] 
implying that 
\begin{align*}
b_{r+s-1}(G,W)=\td(S(W)/K^{r+s-1})\ge\td(S(U)/I^r\otimes S(V)/J^s)
\\=b_r(G/N,U)+b_s(N),V). 
\end{align*}
\end{proof}



\begin{thebibliography}{1}

\bibitem{Cz1} K. Cziszter, The Noether number of the non-abelian group of order
$3p$, Periodica  Math. Hungarica 68 (2014), 150-159.

\bibitem{Cz2} K. Cziszter, K. Cziszter, On the Noether number of $p$-groups,  arXiv:1604.01938 



\bibitem{CzD:3}
K. Cziszter and M. Domokos,   On the generalized Davenport constant and the Noether number, Central European Journal of Mathematics 11 (2013), 1605-1615. 

\bibitem{CzD:1} K. Cziszter and M. Domokos, Groups with large Noether bound,  Ann. Inst. Fourier (Grenoble) 64, no. 3 (2014), 909-944. 



\bibitem{CzD:2}
K. Cziszter and M. Domokos: The Noether number for the groups with a cyclic subgroup of index two,
Journal of Algebra 399 (2014), 546-560. 


  
 \bibitem{CzDG} K. Cziszter, M. Domokos and A. Geroldinger, 
The interplay of invariant theory with multiplicative ideal theory and with arithmetic combinatorics, in: Scott T. Chapman, M. Fontana, A. Geroldinger, B.Olberding (Eds.), Multiplicative Ideal Theory and Factorization Theory, Springer-Verlag, 2016, pp. 43-95.

\bibitem{CzDSz} K. Cziszter, M. Domokos and I. Sz\"oll\H osi, The Noether numbers and the Davenport constants of the groups of order less than 32, 	arXiv:1702.02997.  

\bibitem{domokos-hegedus}
M. Domokos, P. Heged\H{u}s,
Noether's bound for polynomial invariants of finite groups,
Arch. Math. (Basel)  74  (2000), No. 3, 161-167.


\bibitem{fleischmann} P. Fleischmann, The Noether bound in invariant theory of finite groups, Adv. Math. 156 (2000), 23-32. 


\bibitem{fogarty} J. Fogarty, On Noether's bound for polynomial invariants of a finite group, Electron. Res. Announc. Amer. Math. Soc. 7 (2001), 5-7. 
 
\bibitem{hegedus-pyber} P. Heged\H us and L. Pyber, Upper bounds for the Noether number of a finite group, preprint (2015). 

\bibitem{kohls-sezer} M. Kohls and M. Sezer, On the top degree of coinvariants, 
Int. Math. Res. Not. IMRN 2014, no. 22, 6079-6093. 

\bibitem{noether} E. Noether, Der Endlichkeitssatz der Invarianten endlicher Gruppen, 
Math. Ann. 77 (1916), 89-92. 

 
\bibitem{schmid}
B.~J. Schmid, Finite groups and invariant theory, in ``Topics in invariant theory'' (M.-P. Malliavin, ed.), Lecture notes in mathematics, no. 1478,
  Springer, 1989-90, pp.~35-66.

\bibitem{sezer} M. Sezer, Sharpening the generalized Noether bound in the invariant theory of finite groups, J. Algebra 254 (2002), 252-263. 

\end{thebibliography}
\end{document}